\documentclass[oneside,a4paper,reqno,12pt]{amsart}
\usepackage{eurosym}
\usepackage{amssymb}
\usepackage{amsmath}
\usepackage{amsfonts}
\usepackage{amssymb}
\usepackage{amsfonts}
\usepackage{graphicx,color}

\newtheorem{theorem}{Theorem}[section]
\newtheorem{lemma}[theorem]{Lemma}
\newtheorem{proposition}[theorem]{Proposition}
\newtheorem{corollary}[theorem]{Corollary} 
\theoremstyle{remark}
 
\theoremstyle{definition}
\newtheorem{definition}[theorem]{Definition} 
\newtheorem{example}[theorem]{Example} 

\def \({\left( }
\def \){\right) }
\DeclareMathOperator{\R}{\mathbb{R}}
\DeclareMathOperator{\N}{\mathbf{N}}

\begin{document}
\title[Surfaces of osculating circles in Euclidean space]{Surfaces of
osculating circles in Euclidean space}
\author{Rafael L\'{O}PEZ}
\address{Departamento de Geometr\'{\i}a y Topolog\'{\i}a, Universidad de
Granada, 18071 Granada, Spain}
\email{rcamino@ugr.es}
\author{\c{C}etin CAMCI}
\address{Department of Mathematics, Faculty of Sciences and Arts, Onsekiz
Mart University, \c{C}anakkale, TURKEY}
\email{ccamci@comu.edu.tr}
\author{Ali U\c{c}um}
\address{K\i r\i kkale University, Faculty of Sciences and Arts, Department
of Mathematics, K\i r\i kkale-Turkey}
\email{aliucum05@gmail.com}
\author{Kaz\i m\ \.{I}larslan}
\address{K\i r\i kkale University, Faculty of Sciences and Arts, Department
of Mathematics, K\i r\i kkale-Turkey}
\email{kilarslan@yahoo.com}

\begin{abstract}
We introduce a new class of surfaces in Euclidean $3$-space, called  surfaces of osculating circles, using the concept of  osculating circle of a regular curve. These surfaces contain a uniparametric family of planar lines of curvature. In this paper, we  classify those ones that are canal surfaces and  Weingarten surfaces. 
\end{abstract}

\maketitle

\section{Introduction and first results}

In  the classical theory of surfaces in Euclidean space $\R^3$, there are many ways of construction of surfaces with the help of curves. A
clear example are the Darboux surfaces. These surfaces are defined
kinematically as the movement of a curve by a uniparametric family of rigid
motions of $\R^3$. A parametrization of a such surface is $X
(s,t)=A(t)\cdot\alpha (s)+\beta (t)$, where $\alpha $ and $\beta $ are two
space curves and $A(t)$ is an orthogonal matrix. The class of Darboux
surfaces includes translation surfaces, circular surfaces, ruled surfaces,
surfaces of revolution and Monge surfaces. Since these surfaces are
generated by curves followed of rigid motions of $\R^3$, they have received a special attention   in computer aided geometric design (CAGD) of geometric models in
architecture, engineering or physics thanks to  their ease of computation and calculation (\cite{Pot}).

From the geometric viewpoint, an important work is the
classification of these surfaces according its Gaussian curvature and its
mean curvature. For example,  ruled surfaces
and surfaces of revolution with constant (Gaussian or mean) curvature are well known. In contrast, the full classification of translation surfaces with constant
curvature was an open problem until very recently: see \cite{ha,hl1,hl2,lp}.

In this paper, we define a new  class of surfaces by means of the concept of   osculating
circle of a regular curve in  $\R^3$. Recall that if $C$ is a planar regular curve and $p\in C$ is a
given point, the osculating circle at $p$ is the tangent circle to $C$ at $p$ and
with the same curvature of $C$ at $p$. If $\alpha =\alpha (s)$ is a parametrization by
arc length and $\{T,N\}$ is its Frenet frame, the osculating circle
at a point $s$ with non-vanishing curvature $\kappa (s)$ parametrizes as
\begin{equation}\label{uu}
u\mapsto \alpha (s)+\frac{1}{\kappa (s)}N(s)+\frac{1}{\kappa (s)}\left(
\sin{u}\, T(s) -\cos{u}\, N(s)\right) .
\end{equation}%
Let us notice that if $u=0$, then the point of the circle is just the point $%
\alpha (s)$. In the case that $C$ is a spatial curve, the definition of the
osculating circle coincides with \eqref{uu}, observing that this circle is included in the
osculating plane at the point $\alpha(s)$.

We are in conditions to introduce the new object of study.

\begin{definition}
Let $\alpha\colon I\subset \R\to\R^3$ be a regular curve parametrized by arc length with non-zero
curvature $\kappa$. Let $r(s)=1/\kappa(s)$ be the radius of curvature. The surface of osculating circles is the parametrized
surface $X\colon I\times\R\to\R^3$ defined by
\begin{equation}\label{sur}
\begin{split}
X(s,u)& =\alpha(s)+r(s)N(s)+r(s)\left(\sin u\, T(s)-\cos u\, N(s)\right) \\
&=\alpha(s)+r(s)\left(\sin{u}\, T(s)+(1-\cos{u})N(s)\right).  
\end{split}%
\end{equation}%
 The curve $\alpha $ is called the generator of the
surface and the parametric $s$-curves, $u\mapsto X(s,u)$ are called parallels. 
\end{definition}

A special case of surfaces of osculating circles occurs when the generator is a planar curve because in such a case, the surfaces must be an open subset of a plane.

\begin{proposition}\label{pr-plane}
If the generator of a surfaces $S$ of osculating circles is contained in a plane $P$, then $S$ is a subset of $P$.
\end{proposition}

\begin{proof} Suppose that  the generator $C$ is contained in the plane $P$ of equation $\langle {\bf x}-p_0,\vec{v}\rangle=0$, $p_0\in\R^3$, $|\vec{v}|=1$. Here $\langle,\rangle$ is the Euclidean metric of $\R^3$. Then the binormal vector $B(s)$ of $C$ is $\pm \vec{v}$. From \eqref{sur},  it is immediate that 
$$\langle X(s,u)-p_0,\vec{v}\rangle=\langle\alpha(s)-p_0,\vec{v}\rangle=0.$$
\end{proof}

We show some examples of surfaces of osculating circles. The generator of the first example is a helix, a curve with constant curvature and constant torsion. In the second example, the generator is a cubic curve and finally, in the third example we show a compact surface whose generator is a closed curve.

\begin{example}
\label{ex1} Let $\alpha$ be the circular helix
\begin{equation*}
\alpha (s) =\left( 2\cos \frac{s}{\sqrt{5}},2\sin \frac{s}{\sqrt{5}},\frac{s%
}{\sqrt{5}}\right).
\end{equation*}%
The curvature is $\kappa =2/5$ and the torsion is $\tau =1/5$.  The tangent vectors $T(s)$ and the normal vectors $N(s)$ are
\begin{eqnarray*}
T &=&\frac{1}{\sqrt{5}}\left( -2\sin \frac{s}{\sqrt{5}},2\cos \frac{s}{\sqrt{5}},1\right) , \\
N &=&\left( -\cos \frac{s}{\sqrt{5}},-\sin \frac{s}{\sqrt{5}},0\right).
\end{eqnarray*}%
The parametrization of the surface of osculating circles   is
\begin{equation*}
X(s,u) =\left(
\begin{array}{l}
\frac{1}{2} \left(\cos {\frac{s}{\sqrt{5}}} (5 \cos{u}-1)-2 \sqrt{5} \sin {%
\frac{s}{\sqrt{5}}} \sin{u}\right) \\
\frac{1}{2} \sin {\frac{s}{\sqrt{5}}} (5 \cos{u}-1)+\sqrt{5} \cos {\frac{s}{%
\sqrt{5}}} \sin{u} \\
\dfrac{2 s+5 \sin{u}}{2 \sqrt{5}}%
\end{array}%
\right).
\end{equation*}%
The graphic of the surface of osculating circles $X(s,u) $ appears  in Figure \ref{fig1}.
\begin{figure}[hbtp]
\begin{center}
\includegraphics[width=.35\textwidth]{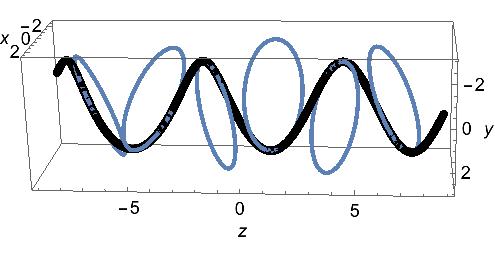} \quad %
\includegraphics[width=.4\textwidth]{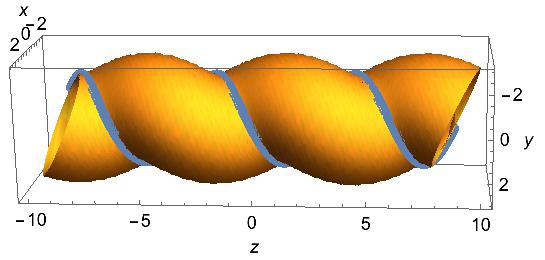}
\end{center}
\caption{Example  \ref{ex1}. Left: The generator $\alpha$
(black) together with some parallels. Right: the surface of osculating
circles. }\label{fig1}
\end{figure}
\end{example}

\begin{example}
\label{ex2} Let $\alpha$ be the cubic $\alpha(s)=(s,s^2/2,s^3/3)$. Let us
notice that this curve is not parametrized by arc length. The radius of
curvature is
\begin{equation*}
r(s)=\frac{\left(1+s^2+s^4\right)^{3/2}}{1+s^2}
\end{equation*}
and the tangent vectors and the normal vectors are, respectively, 
\begin{equation*}
\begin{split}
T(s)&=\frac{1}{\sqrt{1+s^2+s^4}}(1,s,s^2), \\
N(s)&=\frac{1}{(1+s^2)\sqrt{1+s^2+s^4}}\left(s(1+2s^2),1-s^4,s(2+s^2)\right).
\end{split}%
\end{equation*}
Thus the surface of osculating circles is (see Figure \ref{fig2}, left)
\begin{equation*}
X(s,u) =\left(
\begin{array}{l}
\frac{\left(s^4+s^2+1\right) \left(\left(2 s^3+s\right) \cos{u}+(s^2+1) \sin{%
u}\right)-s^3 \left(2 \left(s^4+s^2\right)+1\right)}{(s^2+1)^2} \\
\frac{2 \left(s^6-1\right) \cos{u}-2 s^6+s^4+s^2+2 \left(s^5+s^3+s\right)
\sin{u}+2}{2 (s^2+1)} \\
\frac{s \left(4 s^6+11 s^4+10 s^2+3 \left(s^4+s^2+1\right) \left(s (s^2+1)
\sin{u}-(s^2+2) \cos{u}\right)+6\right)}{3 (s^2+1)^2}%
\end{array}%
\right).
\end{equation*}
\end{example}

Finally, we show an example of a compact surface of osculating circles.

\begin{example}
\label{ex3} Consider the generator
\begin{equation*}
\alpha(s)=(\cos (s) (4+\cos (2 s)),\sin (s) (4+\cos (2 s)),\sin (2 s)).
\end{equation*}
This curve is closed and simple and it is contained in a torus of inner and outer radii $1$
and $4$, respectively. The surface of osculating circles whose generator is $\alpha$ has a cumbersome parametrization but its picture appears in Figure \ref{fig2}, right.
\begin{figure}[hbtp]
\begin{center}
\includegraphics[width=.35\textwidth]{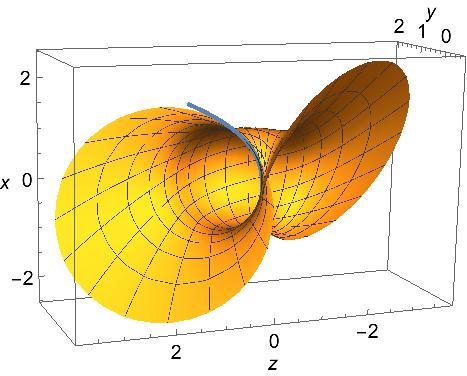}\quad %
\includegraphics[width=0.4\textwidth]{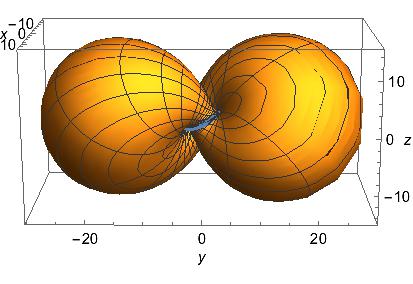}
\end{center}
\caption{Surfaces of osculating circles. Left: Example \ref{ex2}.
Right: Example \ref{ex3}. }
\label{fig2}
\end{figure}
\end{example}

The organization of this paper is the following. In Section \ref{sec2}, we study some basic properties of the surfaces of osculating circles, such as, its 
regularity and umbilicity. For this, we compute the first and the second fundamental form of the surface. We also discuss the necessary and sufficient condition   to be a canal surface. We prove that the surface of osculating
circles is a canal surface if and only if the curvature of the generator is constant (Theorem \ref{pr-canal}). Curves in $\R^3$ with constant curvature are called Salkowski curves  (\cite{mo,sa}).  In Section \ref{sec3}, we classify  the surfaces of osculating circles of Weingarten type, proving the the generator is a   Salkowski curve (Theorem \ref{th-weingarten}). Finally,  we prove that open subsets of planes and  spheres are the only surfaces of osculating circles with constant Gaussian curvature or constant mean curvature (Theorem \ref{t-mean}). 

\section{The curvature of the surfaces of osculating circles}\label{sec2}

We begin this section analyzing   in what points the surface of osculating circles is regular.

\begin{proposition}\label{pr-regular}
The set of non-regular points of a surface of osculating circles is formed
by the generator $\alpha$ and the set $\{X(s,u):r'(s)=\tau(s)=0, u\in\R\}$.  
\end{proposition}

\begin{proof}
We compute the partial derivatives of $X(s,u)$. As usually, the subscripts will
indicate the corresponding derivatives and $^{\prime }$ is the
differentiation of   functions of one variable depending on $s$. Using the Frenet
equations of $\alpha$
\begin{equation*}
\begin{split}
X_s&=(r'\sin u+\cos u)T+(r'(1-\cos u)+\sin
u)N+r\tau(1-\cos u)B \\
X_u&= r(\cos u\, T+\sin u\, N).
\end{split}%
\end{equation*}
Therefore,
\begin{equation*}
X_s\times X_u=r(1-\cos u)\left(-r\tau \sin u\, T+r\tau \cos u\, N+r^{\prime
}B\right).
\end{equation*}
Then $(X_s\times X_u)(s,u)=0$ if and only if $\cos{u}=1$, that is, at the points of the generator or when $r'(s)=\tau(s)=0$. 
\end{proof}

From the above proof, we deduce that the unit normal vector field at regular points is
\begin{equation}  \label{normal}
\N(s,u)=\frac{1}{\sqrt{r^2\tau^2+r'^2}}\left(-r\tau \sin u\, T+r \tau
\cos u\, N+r'B\right).
\end{equation}

A particular case of surfaces of osculating circles are the spheres of $\R^3$. 

\begin{proposition}
\label{pr-sphere}  If the generator of a surface of osculating circles is a
spherical curve, then the surface is an open subset of a sphere.
\end{proposition}

\begin{proof}
After a translation, we can assume that the generator $\alpha$ is included
in a sphere of radius $R>0$ centered at the origin.  Thus  $\langle\alpha(s),\alpha(s)\rangle=R^2$, so $\langle T(s),\alpha(s)\rangle=0$ and   $\langle N(s),\alpha(s) \rangle
=-1/\kappa(s) =-r(s)$ for all $s\in I$. From \eqref{sur}
\begin{equation*}
\begin{split}
|X(s,u)|^{2}&=R^{2}+r^2((1-\cos {u})^{2}+\sin^2{u})+2r(1-\cos {u})\langle
N,\alpha \rangle \\
&=R^{2}+r^2((1-\cos {u})^{2}+\sin^2{u})-2r^2(1-\cos {u})=R^{2},
\end{split}%
\end{equation*}%
proving that the surface is included in the same sphere where is contained $%
\alpha$.
\end{proof}

The generator of the surface showed in Example \ref{ex3} is a simple closed curve and the topology of the surface   is clearly a torus.  This  can be generalized in the sense that a compact surface of  osculating circles must be a topological torus.

\begin{proposition} 
  If $M$ is a compact surface of osculating circles, then $M$ is a torus.
\end{proposition}

\begin{proof} Consider a compact surface $M$ of osculating circles parametrized by $X=X(s,u)$.  Define on $M$ the vector field 
\[ V(s,u)=X_u=\sin u\, N(s)+\cos u\, T(s).
\]
 Then $V$ have not zeroes on $M$ so  its  index is $0$. By the Poincar\'e index theorem, the surface must be a topological torus.
\end{proof}

We study when a surface of osculating circles is a canal surface. Recall
that a canal surface is the envelope of the $1$-parameter pencil of spheres $%
\Sigma(s)$, centered at a spine curve $\beta (s)$,  $\Sigma(\mathbf{x};s)= |%
\mathbf{x}-\beta(s)|^{2}-r(s)^{2}=0$ (\cite{Mar}). The envelope condition is given by
\begin{equation}  \label{enve}
\frac{d}{ds} \Sigma(s)=\langle \mathbf{x}-\beta (s) ,\beta^{\prime
}(s)\rangle+r(s)r'(s) =0.
\end{equation}

\begin{theorem}\label{pr-canal}
A surface of osculating circles   is a canal
surface if and only if the generator is a Salkowski curve.
\end{theorem}

\begin{proof}
Let $\beta(s) =\alpha (s) +r(s) N(s)$ be the center of the osculating
circles and consider the $1$-parameter family of spheres $\Sigma(\mathbf{x}%
;s)= |\mathbf{x}-\beta(s)|^{2}-r(s)^{2}=0$. By the definition \eqref{sur} of the surface, it is clear that the
surface is contained in this uniparametric family of spheres. We need to check
the envelope condition \eqref{enve}. Since $\beta'=r'N+r\tau B$, we have
\begin{equation*}
\frac{d}{ds}\Sigma(s)=\langle -r\cos{u}\,N+r\sin{u}\, T,r'N+r\tau B\rangle+rr^{\prime
}=rr'(1-\cos{u})=0
\end{equation*}
for all $s\in I$. Therefore $r'(s)=0$ for all $s\in I$ and this is equivalent to say that $\kappa$ is a nonzero constant function.
\end{proof}

We show an example of a surface of osculating circles
whose generator is a Salkowski curve, in particular, the surface is a canal
surface.

\begin{example}
\label{ex4} Let us consider the Salkowski curve in $\R^3$ parametrized by
\begin{equation*}
\alpha (s) =\left(
\begin{array}{l}
\frac{78s\sqrt{25-s^{2}}\cos \left( \sqrt{26}\arcsin \left( \frac{s}{5}%
\right) \right) +\sqrt{26}(28s^{2}-625)\sin \left( \sqrt{26}\arcsin {\frac{s%
}{5}} \right) }{2860}\\
\frac{\sqrt{26}\left( 625-28s^{2}\right) \cos \left( \sqrt{26}\arcsin{\frac{s%
}{5}} \right) +78s\sqrt{25-s^{2}}\sin \left( \sqrt{26}\arcsin {\frac{s}{5}}
\right) }{2860} \\
\frac{25-2s^{2}}{4\sqrt{26}}%
\end{array}%
\right).
\end{equation*}%
The curvature of $\alpha$ is $\kappa =1$ and the torsion is $\tau =s/\sqrt{25-s^{2}}$. The tangent and normal vectors are 
\begin{eqnarray*}
T &=&\left(
\begin{array}{l}
\frac{-\sqrt{25-s^{2}}\cos \left( \sqrt{26}\arcsin {\frac{s}{5}} \right) }{5}%
-\frac{s\sin \left( \sqrt{26}\arcsin {\frac{s}{5}} \right) }{5\sqrt{26}} \\
\frac{s\cos \left( \sqrt{26}\arcsin{\frac{s}{5}} \right) }{5\sqrt{26}}-\frac{%
\sqrt{25-s^{2}}\sin \left( \sqrt{26}\arcsin \left( \frac{s}{5}\right)
\right) }{5} \\
-\frac{2}{\sqrt{26}}%
\end{array}%
\right) , \\
N &=&\frac{1}{\sqrt{26}}\left(5 \sin \left( \sqrt{26}\arcsin {\frac{s}{5}} \right) ,-5\cos \left( \sqrt{26}\arcsin {\frac{s}{5}} \right) ,-1\right).
\end{eqnarray*}%
So the parametrization of the surface of osculating circles is $X(s,u)=(X_1(s,u),X_2(s,u),X_3(s,u))$   given by%
\begin{eqnarray*}
X _{1}(s,u) &=&\frac{26\sqrt{25-s^{2}}\cos \left( \sqrt{26}\arcsin {\frac{s%
}{5}}\right) \left( 3s-22\sin u\right) }{2860} \\
&&-\frac{\sqrt{26}\sin \left( \sqrt{26}\arcsin {\frac{s}{5}} \right) \left(
75-28s^{2}+550\cos u+22s\sin u\right) }{2860},
\end{eqnarray*}%
\begin{eqnarray*}
X _{2}(s,u) &=&\frac{26\sqrt{25-s^{2}}\sin \left( \sqrt{26}\arcsin{\frac{s}{%
5}} \right) \left( 3s-22\sin u\right) }{2860} \\
&&+\frac{\sqrt{26}\cos \left( \sqrt{26}\arcsin {\frac{s}{5}} \right) \left(
75-28s^{2}+550\cos u+22s\sin u\right) }{2860}, \\
X _{3}(s,u) &=&\frac{21-2s^{2}+4\cos u-4s\sin u}{4\sqrt{26}}.
\end{eqnarray*}%
The graphic of  $X(s,u) $ is depicted in Figure %
\ref{fig3}.
\begin{figure}[hbtp]
\begin{center}
\includegraphics[width=.5\textwidth]{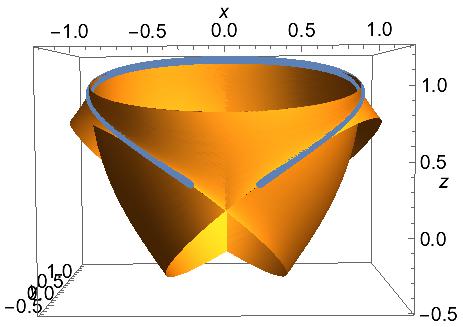}
\end{center}
\caption{ A surface of osculating circles whose
generator is a Salkowski curve: Example \ref{ex4}. }
\label{fig3}
\end{figure}
\end{example}

The property that parallel are circles implies remarkable properties on them.  

\begin{proposition} 
Parallels of  a  surface of osculating
circles are lines of curvature. In case that $r'(s_0)=0$, then the parallel is also a geodesic.  
\end{proposition}

\begin{proof} Consider a parallel $\gamma(u)=X(s,u/r)$ which is parametrized by arc length. Then $\gamma'(u)=\sin{u}\,T+\cos{u}\, N$. On the other hand, and thanks to \eqref{normal}, the derivative of the normal $\N(\gamma(u))$ along the curve is, 
$$\frac{d}{ds}\N(\gamma(u))=-\frac{\tau}{\sqrt{r^2\tau^2+r'^2}}(\cos{u}\,T+\sin{u}\,N)=-\frac{\tau}{\sqrt{r^2\tau^2+r'^2}}\gamma'(u).$$
This proves that $\gamma$ is a line of curvature. Moreover, the normal curvature is $\kappa_n(u)=\tau/\sqrt{r^2\tau^2+r'^2}$, being constant along $\gamma$.

We compute the geodesic curvature $\kappa_g$. Then we have
$$\kappa_g(u)=-\langle\gamma'\times\gamma'',\N(\gamma)\rangle=\frac{r'}{r\sqrt{r^2\tau^2+r'^2}}.$$
In particular, $\kappa_g=0$ if and only if $r'=0$.
\end{proof} 

One can also prove this result observing that each parallel is contained in
the plane $\{T,N\}$, which is orthogonal to $B(s)$ and the angle that makes $%
B(s)$ with the normal $\N(s,u)$ to the surface is constant, so we can apply
the Joachimstahl's theorem. By this proposition, notice that  
the surface contains a family of lines of curvature which are planar. Thus surfaces of osculating circles form part of an interesting family of surfaces having one or two uniparametric families of planar lines of curvature. For example, this property is shared by the surfaces of revolution and the Monge surfaces (\cite{bg}).

We now compute the Gaussian curvature and the mean curvature of the surface. Firstly, we calculate the coefficients $\{E,F,G\}$ and $\{e,f,g\}$ of the first and the
second fundamental form, respectively. For the first
fundamental form, we have
\begin{equation*}
\begin{split}
E& =\frac{1}{2}\left( 2+8r^2\tau^2\sin^{4} \frac{u}{2} +4r'\sin
u+4r'^2(1-\cos u)\right), \\
F& =r\left( 1+r'\sin{u}\right), \\
G& =r^2.
\end{split}%
\end{equation*}
In particular,
\begin{equation}\label{ww}
\sqrt{EG-F^{2}}=2r \sqrt{r^2\tau^2+r'^2}\sin ^{2}{\frac{u}{2}}.
\end{equation}%
On the other hand,
\begin{equation*}
\begin{split}
\mbox{det}(X_{s},X_{u},X_{ss})& =r\sin^{2}{\frac{u}{2}} \Big(2\tau
r'\left( \sin{u}-2(\cos{u}-1)r'\right) -4r^2\tau
^{3} \cos{u}\sin^{2}{\frac{u}{2}} \\
& +2r\left( \kappa \tau  ( 1+r'\sin{u}) -(\cos {u}-1)\left( r'\tau'-\tau r''\right) \right) %
\Big),\\
\mbox{det}(X_{s},X_{u},X_{su})& =2r^2\tau ( 1+r'\sin{u})\sin ^{2}{\frac{u}{2}}, \\
\mbox{det}(X_{s},X_{u},X_{uu})& =r^{3}\tau (1-\cos{u}).
\end{split}%
\end{equation*}%

Then
\begin{eqnarray*}
e &=&\frac{1}{\sqrt{r^2\tau^2+r'^2}}\Big(\tau r'\left( \sin{%
u}-2(\cos{u}-1)r'\right) -2r^2\tau^3\cos{u}\sin ^{2}{\frac{u}{2}} 
\\
&&+ \tau (1+r'\sin{u}) -r(\cos{u}-1)\left( r''-\tau
r''\right) \Big), \\
f &=&\frac{r\tau \left(1+r'\sin{u}\right) }{\sqrt{%
r^2\tau^2+r'^2}}, \\
g &=&\frac{r^2\tau }{\sqrt{r^2\tau^2+r'^2}}.
\end{eqnarray*}

From the calculation of the first and the second fundamental forms, we find   the Gaussian curvature $K$ and the mean curvature $H$.

\begin{proposition}
The Gaussian curvature $K$ and the mean curvature $H$ of a surface of
osculating circles is  
\begin{equation}
K= \frac{\tau \left( r^2\tau^3\cos u+r'^2\tau ( \cos u-1) +r\left(
\tau r''-r'\tau'\right) \right) }{\left(
\cos u-1\right) \left( r^2\tau ^{2}+r'^2\right) ^{2}}  \label{kk}
\end{equation}%
and
\begin{equation}
H= -\frac{r^2\tau^3\left( 2\cos u-1\right) +2 r'^2\tau ( \cos u-1)
+r ( \tau r''-r'\tau') }{2(1-\cos {u}%
)\left( r^2\tau ^{2}+r'^2\right) ^{3/2}}.  \label{hh}
\end{equation}
\end{proposition}

As a consequence, we find the umbilical points. First, we need the next
auxiliary lemma (e.g. \cite[p. 25]{dc}).

\begin{lemma}
\label{le1} Let $\alpha=\alpha(s)$ be a curve in $\R^3$ parametrized by arc
length with $\tau\not=0$ and   $\kappa'\not=0$ for all $s\in I$. Then $\alpha $ is included in a sphere of radius $R$  if and only if the function $%
r^2+ r'^2/\tau^2=R^2$. Equivalently, it holds the identity
\begin{equation}  \label{condition-spherical}
r''\tau-r'\tau'+r\tau^3=0.
\end{equation}
\end{lemma}

\begin{corollary}\label{cor-um}
The umbilical points of a surface $X(s,u)$ of osculating circles are those
points where \eqref{condition-spherical} holds.
\end{corollary}

\begin{proof}
An umbilical point  is characterized by the identity $H^{2}-K=0$. Then \eqref{kk} and \eqref{hh} lead to
we have
\begin{equation*}
H^{2}-K=\frac{r^2(r''\tau-r'\tau'+r\tau^3)^2}{%
4\left( 1-\cos u\right) ^{2}\left( r^2\tau ^{2}+r'^2\right) ^{3}},
\end{equation*}%
proving the result.
\end{proof}

In particular, those points $s=s_0$ where \eqref{condition-spherical} is fulfilled, the
parallel $u\mapsto X(s_0,u)$ is formed by umbilical points. Comparing with Lemma \ref{le1}, Corollary \ref{cor-um} says that in case that $\tau\not=0$ and $\kappa'\not=0$,  the umbilical points correspond with those points of the generator that are close to be contained in a sphere.


\section{Surfaces of osculating circles of Weingarten type}\label{sec3}


A surface in Euclidean space is said to be a Weingarten surface if there is
a nontrivial smooth functional relation $W(K,H)=0$. In the general
case, it is still an open question the full classification of Weingarten surfaces. A particular case of Weingarten surfaces
is when the relation $W(K,H)$ is linear, so $aH+bK+c=0$ for $a$, $b$, and $c$
are not all zero real numbers. These surfaces are called linear Weingarten
surfaces. A linear Weingarten surface satisfying $aH+bK+c=0$ is elliptic (resp.
hyperbolic, parabolic) if $a^2-4bc>0$ (resp. $a^2-4bc<0$ and $a^2-4bc=0$) (\cite{ga,lo2}). In this section, we study under what conditions, surfaces of
osculating circles are Weingarten surfaces.

\begin{theorem}\label{th-weingarten}
If a surface of osculating circles is a Weingarten surface, then  it is an open subset of a plane, of a sphere or the generator is a Salkowski curve.
In the latter case, the surface is a linear Weingarten surface of parabolic type and the
Weingarten relation $W (K,H)=0$ is
\begin{equation}\label{h-k}
H-\frac{r}{2}K-\frac{1}{2r}=0.
\end{equation}
\end{theorem}

\begin{proof} If we differentiate the Weingarten relation $W(K,H)=0$ with respect to $s$ and $u$, and using the chain rule, we have $K_{s}H_{u}-K_{u}H_{s} =0$. From $\left( \ref{kk}\right) $ and $\left( \ref{hh}\right) $, we find
\[K_{s}H_{u}-K_{u}H_{s}=\frac{r^2r'\left( r \tau^3-r'\tau'+ r''\tau\right) ^{3}\sin u}{2\left( \cos u-1\right) ^{3}\left(
r^2\tau^2+r'^2\right) ^{9/2}}.\]
Since $\sin{u}\not=0$ by regularity of the surface (Proposition \ref{pr-regular}), the surface is of Weingarten type if and only if $r'\left( r \tau^3-r'\tau'+ r''\tau\right)=0$. 

Suppose that there is $s_0\in I$ such that $r'(s_0)\not=0$. Then we have $r^2\tau^3-r'\tau'+
r''\tau=0$ in
an interval around $s_0$. If the torsion $\tau$ is constantly $0$, then the generator is a planar curve and the surface is a subset of a plane by Proposition \ref{pr-plane}. Otherwise, the equation $r^2\tau^3-r'\tau'+
r''\tau=0$  implies that the generator is a spherical
curve (Lemma \ref{le1}) and the surface is an open subset of a sphere
(Proposition \ref{pr-sphere}). 

Finally, if $r'(s)=0$ for all $s\in I$, then  the curvature of $\alpha$ is constant and $\alpha$ is a Salkowski curve. In this case, if $r(s)$ is a constant function, say $r(s)=r$, we deduce from \eqref{kk} and %
\eqref{hh}
\begin{equation}\label{hh-kk}
H=-\frac{2\cos{u}-1}{2r(1-\cos{u})},\quad K=-\frac{\cos{u}}{r^2(1-\cos{u})}.
\end{equation}
Now \eqref{h-k} is immediate being a linear Weingarten relation $aH+bK+c=0$. Furthermore,  $a=1$, $b=-r/2$ and $c=-1/(2r)$, concluding $a^2-4bc=0$ and  the linear Weingarten relation is parabolic. 
\end{proof}

By Proposition \ref{pr-canal}, the last case implies that the surface is a canal surface and, in consequence, in the class of surfaces of osculating circles, canal surfaces coincide with Weingarten surfaces. Moreover,  a simple computation gives
that the principal curvatures are
\begin{equation*}
\kappa_1=\frac{1}{r}=\kappa,\quad \kappa_2=-\frac{\cos{u}}{r(1-\cos{u})},
\end{equation*}
where $\kappa_1$ is just the curvature of the generator. 

An interesting family of surfaces in Euclidean space are the surfaces   with constant Gaussian curvature or constant mean curvature. The classification of these surfaces in the class of surfaces of osculating circles follows from  general results because
surfaces of osculating circles are surfaces formed by a uniparametric family
of circles, in particular, these surfaces fall in the class of cyclic
surfaces. Cyclic surfaces with constant Gaussian curvature $K=c$ or constant
mean curvature $H=c$, are planes, spheres (where the circles of the
foliation are not necessarily included in parallel planes), surfaces of
revolution, cones ($K=0$) or the Riemann minimal examples ($H=0$). See \cite%
{lo,ni1,ni2}. However, given a parametrization \eqref{sur} of a surface of osculating circles it is very difficult to deduce that this surface is one of the above cases. We obtain the classification thanks to the above computations of $K$ and $H$.

\begin{theorem}\label{t-mean}
Open subsets of planes and spheres are the only surfaces of osculating circles with
constant Gaussian curvature or constant mean curvature.
\end{theorem}

\begin{proof}
\begin{enumerate}
\item Suppose that $K$ is a constant $c$. Then \eqref{kk} becomes
\begin{equation*}
\tau \left( r^2\tau^3\cos u+\tau r'^2( \cos u-1) +r\left( \tau
r''-\tau'r'\right) \right) -c\left( \cos
u-1\right) ( r^2\tau^2+r'^2) ^{2}=0.
\end{equation*}%
This gives an polynomial equation on $\{1,\cos{u}\}$ which are linearly
independent. Thus all coefficients, which are functions on $s$, must vanish,
obtaining two equations. If $c=0$, the coefficient of $\cos {u}$ is $\tau
(r^2\tau^2+r'^2)$, obtaining $\tau(s) =0$ for all $s\in I$. This implies that the surface is a subset of a plane (Proposition \ref{pr-plane}).   If $
c\not=0$, both equations are
\begin{equation*}
r\tau(r''\tau-r'\tau')-r'^2\tau^2+c(r^2\tau^2+r'^2)^2=0.
\end{equation*}
\begin{equation*}
(r^2\tau^2+r'^2)(\tau^2 -c(r^2+\tau^2r'^2))=0.
\end{equation*}
The first equation implies that $\kappa$ is not a constant function. From the second equation, and by regularity, $\tau^2 =c(r^2\tau^2+r'^2) $. In particular, $\tau\not'=0$. Substituting in the first equation, we have
\begin{equation*}
r(r''\tau-r'\tau'+r\tau^3)=0.
\end{equation*}%
Hence the result follows from Lemma \ref{le1} and Proposition \ref{pr-sphere}.

\item Assume that $H$ is constant. Firstly assume that $H=0.$ Then \eqref{hh}
gives
\begin{equation*}
r^2\tau^3\left( 2\cos u-1\right) +2\tau r^2\left( \cos u-1\right) +r\left(
\tau r''-\tau'r'\right) =0.
\end{equation*}%
The coefficient of $\cos u$ is $2\tau \left( r^2\tau^2+r'^2\right)$.
Hence $\tau(s)=0$ for all $s\in I$ and the surface is a subset of a plane (Proposition \ref{pr-plane}).  Now assume that $H=c\neq 0 $. Then 
\begin{equation*}
r^2\tau^3\left( 2\cos u-1\right) +2\tau r'^2\left( \cos u-1\right)
+r\left( \tau r''-\tau'r'\right) +2c\left(
1-\cos u\right) \left( r^2\tau^2+r'^2\right) ^{3/2}=0.
\end{equation*}%
The coefficients of $\{1,\cos{u}\}$ must vanish, so
\begin{equation*}
\sqrt{r'^2+r^2\tau^2}(2c(r^2\tau^2+r'^2)-\tau)-r^{\prime
2}\tau+r(r''\tau-r'\tau')=0
\end{equation*}
\begin{equation*}
-2( r'^2+r^2\tau^2) \left( \tau -c\sqrt{r'^2+r^2\tau^2}%
\right)=0.
\end{equation*}
By regularity of the surface, we deduce from the second equation
\begin{equation*}
\tau -c\sqrt{r^2\tau^2+r'^2}=0
\end{equation*}%
Substituting in the first equation, we have $r(r''\tau-r'\tau'+r\tau^3)=0$. In particular, $\kappa$ is not a constant function and the result is again a consequence of Lemma \ref{le1} and Proposition \ref{pr-sphere}.  
 
\end{enumerate}
\end{proof}


\begin{thebibliography}{99}
\bibitem{bg} Brander, D., Gravesen, J., Monge surfaces and planar geodesic
foliations. J. Geom. 109 (2018), no. 1, Paper No. 4, 14 pp.

\bibitem{dc} do Carmo, M.P., Differential Geometry of Curves and Surfaces.
Prentice-Hall, Inc., Englewood Cliffs, N.J. (1976).

\bibitem{ga}  G\'alvez, J. A.,  Mart\'{\i}nez, A.,     Mil\'an, F., Linear Weingarten surfaces in $R^3$, Monatsh. Math. 138 (2003), 133--144.
 
\bibitem{ha} Hasanis, T.,   A characteristic property of circular cylinders. J. Geom. 112 (2021), no. 3, Paper No. 36.

\bibitem{hl1} Hasanis T. and L\'opez R., Classification and construction of
minimal translation surfaces in Euclidean space. Results Math. 75 (2020),
no. 1, Paper No. 2, 22 pp.

\bibitem{hl2} Hasanis T. and L\'opez R., Translation surfaces in Euclidean
space with constant Gaussian curvature. Comm. Anal. Geom., to appear.

\bibitem{lo} L\'opez, R., Cyclic surfaces of constant Gauss curvature.
Houston Math. J. 27 (2001), 799--805.

\bibitem{lo2} L\'opez, R.,  Linear Weingarten surfaces in Euclidean and hyperbolic space, Mat. Contemp. 35 (2008), 95--113.
\bibitem{lp} L\'opez, R., Perdomo, O.,   Minimal translation surfaces in Euclidean space, J. Geom. Anal. 27 (2017),  2926--2937.

\bibitem{mo} Monterde, J., Salkowski curves revisited: a family of curves
with constant curvature and non-constant torsion, Comput. Aided Geom. Design
26 (2009), 271--278.

\bibitem{ni1} Nitsche J. C. C., Lectures on Minimal Surfaces. Cambridge
Univ. Press, Cambridge, 1989.

\bibitem{ni2} Nitsche J. C. C., Cyclic surfaces of constant mean curvature.
Nachr. Akad. Wiss. Gottingen Math. Phys. II 1 (1989) 1--5.

\bibitem{Mar} Peternell M., Pottmann H., Computing rational
parametrizations of canal surfaces, J. Symbolic Computation (1997) 23,
255--266

\bibitem{Pot} Pottmann H., Asperl A., Hofer M,  Kilian A., Architectural
Geometry. Bentley Institute Press, Waltham,MA (2007)

\bibitem{sa} Salkowski, E., Zur Transformation von Raumkurven Math. Ann. 66
(1909), 517--557
\end{thebibliography}
\end{document}